\documentclass[10pt]{amsart}
\usepackage{amsmath}
\usepackage{amsfonts}
\usepackage{amssymb}
\usepackage{amsthm}
\usepackage{mathrsfs}       
\usepackage{dsfont}
\usepackage[colorlinks=true]{hyperref}

\usepackage{fullpage}

\newcommand{\udd}{\mathrm{d}}
\newcommand{\ud}{\,\mathrm{d}}

\newcommand{\ovl}[1]{\overline{#1}}

\newcommand{\R}{\mathbb{R}}
\newcommand{\N}{\mathbb{N}}
\newcommand{\T}{\mathbb{T}}

\newcommand{\C}{\mathscr{C}}

\usepackage{enumitem}

\usepackage{cite}
\usepackage{cleveref}

\title[Degenerate Hypocoercivity on the Torus]{Quantitative rates of convergence to equilibrium for the Degenerate linear Boltzmann equation on the Torus}

\author[J.~Evans]{Josephine Evans$^{\ast}$}
\thanks{$^{\ast}$ \emph{Warwick Mathematics Institute}}
\begin{address}{Josephine Evans,
Warwick Mathematics Institute, Zeeman building, University of Warwick, CV4 7AL}
\end{address}

 \author[I.~Moyano]{Iv\'an Moyano$^{\ast \ast}$}
\thanks{$^{\ast\ast}$ \emph{LJAD, Universit\'e C\^ote d'Azur}}
\begin{address}{Iv\'an Moyano,
LJAD, Université de Nice Sophia Antipolis, 06108 Nice Cedex 02, France.}
\end{address}

\keywords{Convergence to equilibrium; Hypocoercivity; Linear Boltzmann Equation; Degenerate Hypocoercivity, Geometric Control Condition. MS classifications 35B40, 35Q49, 35Q70}

\newtheorem{thm}{Theorem}
\newtheorem{defn}{Definition}
\newtheorem{lemma}{Lemma}
\newtheorem{prop}{Proposition}
\newtheorem*{remark}{Remark}

\newtheorem{cor}{Corollary}
\begin{document}
\maketitle
\begin{abstract}
We study the linear relaxation Boltzmann equation on the torus with a spatially varying jump rate which can be zero on large sections of the domain. In \cite{BS13} Bernard and Salvarani showed that this equation converges exponentially fast to equilibrium if and only if the jump rate satisfies the geometric control condition of Bardos, Lebeau and Rauch \cite{BLR91}. In \cite{HL15} Han-Kwan and L\'{e}autaud showed a more general result for linear Boltzmann equations under the action of potentials in different geometric contexts, including the case of unbounded velocities. In this paper we obtain quantitative rates of convergence to equilibrium when the geometric control condition is satisfied, using a probabilistic approach based on Doeblin's theorem from Markov chains. 
\end{abstract}

\tableofcontents

\section{Introduction and Main Results}
In this article, we study the linear Boltzmann equation in the phase space $\Omega\times V$ (here we study models where $V = \mathbb{R}^d$ and where $V$ is a subset of $\mathbb{R}^d$ not containing zero), i.e., the system 
\begin{equation} 
\left\{ \begin{array}{ll}
\partial_t f + v \cdot \nabla_x f - \nabla_x W(x) \cdot \nabla_v f = \C(f),  & \textrm{in } (0,T) \times \Omega \times V, \\
f|_{t=0} = f_0, & \textrm{in } \Omega \times V,
\end{array} \right.
\label{eq:main} 
\end{equation} where the density function, $f = f(t,x,v)$, undergoes the action of the potential $W = W(x)$ and the collision term 
\begin{equation*}
\C(f) := \sigma(x) \int_V \left( p(v,v') f(v') - p(v',v) f(v) \right) \ud v',
\end{equation*} for some $\sigma \in C^0(\Omega)$, assumed to be non-negative, and $p \in C^1(V \times V)$ a transition kernel ($\int_V p(v,v') \mathrm{d}v =1$). Physically we can think of (\ref{eq:main}) as modelling a radiative transfer system where different parts of the space may have different transparencies, according to the scattering function $p = p(v,v')$. When $\sigma = \sigma(x)$ is a positive constant, (\ref{eq:main}) is the linear relaxation equation, linear BGK equation or linear Boltzmann equation. \par 

In this work we set $\Omega = \T^d$, the $d$-dimensional torus, with the usual identification
\begin{equation}
\T^d = \R^d /\mathbb{Z}^d.
\label{eq:torus}
\end{equation} Let $u_1$ be the uniform measure in $\T^d$. We stress that it is not possible (or necessary) to write down an explicit equilibrium for all of the examples given.

In the non-degenerate case $\sigma > 0$, the study of the trend to equilibrium of solutions to system (\ref{eq:main}) has been the object of many publications, using techniques such as hypocoercivity (see Section \ref{sec:main results} for details). In the degenerate case $\sigma \geq 0$, the problem of characterising the trend to equilibrium is deeply connected to the structure of the phase space $\T^d \times V$ and the geometry of the set $\left\{  \sigma > 0 \right\}$, as (\ref{eq:main}) reduces to a transport equation outside this region. In \cite{BS13b} Bernard and Salvarani showed that exponential convergence towards equilibrium cannot hold in general. On the other hand, the same authors proved in \cite{BS13} that the solutions to (\ref{eq:main}) with $\Omega \times V = \T^d \times \mathbb{S}^{d-1}$ and $W = 0$ converge to equilibrium exponentially in $L^1$ if and only if the support of $\sigma$ satisfies the geometric control condition (GCC for short), inspired from \cite{BLR91,LebeauOndes} and characterized in the following way.  

\begin{defn} \label{def:GCC1} 
The function $\sigma$ satisfies the Geometric Control Condition (GCC) if there exists $T = T(\sigma) > 0, \kappa >0$ such that
\begin{equation} 
\inf_{(x,v) \in \T^d \times V}\int_0^T \sigma(x-vt) \ud t \geq \kappa. 
\label{eq:GCC}
\end{equation}
\end{defn}

The case $W \not = 0$ and $\sigma \geq 0$ has been analysed by Han-Kwan and L\'eautaud in \cite{HL15}, where the action of the potential may generate many different dynamics. Considering the characteristic flow
\begin{equation}
\Phi_t (x,v) = \left( \Phi^X_t(x,v), \Phi^V_t(x,v) \right), \qquad t \in \R, 
\label{eq;characteristic flow}
\end{equation} where, for $(x,v)\in \T^d \times V$ given, $(\Phi^X_t,\Phi^V_t) = \left( \Phi^X_t(x,v), \Phi^V_t(x,v) \right) $ solve the characteristic equations
\begin{equation}
\begin{array}{ll}
 \frac{\ud }{\ud t}  \Phi^X_t = \Phi^V_t, &      \Phi^X_0=x, \\
  \frac{\ud }{\ud t}\Phi^V_t = - \nabla_x W(\Phi^X_t), & \Phi^V_0=v.
\end{array}
\label{eq:characteristic equations}
\end{equation}  The authors adapt the Geometric Control Condition to the action of a potential $W$ in the following way.

\begin{defn} \label{def:GCC2}
Let $W \in C^1(\T^d)$ be a given potential. A function $\sigma \in L^{\infty}(\T^d)$ satisfies the Geometric Control Condition (GCC) if there exist $T = T(\sigma,W)>0, \kappa > 0$ such that 
\begin{equation}
\inf_{(x,v) \in \T^d \times V} \int_0^T \sigma(\Phi^X_{-t}(x,v)) \ud t \geq \kappa, 
\end{equation} where $(\Phi_{t})_{t\geq \in \R}$ is the flow defined by (\ref{eq;characteristic flow}) and (\ref{eq:characteristic equations}).
\end{defn}

This definition is again inspired from the study of the controllability of the wave equation in \cite{BLR91,LebeauOndes} (see Section \ref{sec:main results} for more details). In this context, Han-Kwan and L\'eautaud give in \cite{HL15} conditions linking the collision kernel and the potential which imply either convergence to a steady state or exponential convergence to a steady state.  Let us mention that the results in \cite{HL15} are much more general (see Section \ref{sec:main results}) than the setting presented here.  \par

The methods developed in the works \cite{BS13, BS13b, HL15} do not yield constructive convergence rates for the trend to equilibrium. The goal of the present work is to obtain quantitative rates using different methods, inspired in tools from Markov chains.

\subsection{Main results}

We shall consider the following two regimes 
\begin{description}
\item[(R1)] $W=0$, $V$ is a subset of $\mathbb{R}^d$ which is bounded away from the origin (so that it is possible to satisfy the GCC), and there exist $v_* \in \mathbb{R}^d$ and $r_0, \gamma$ strictly positive constants such that
\[ p(v,v') \geq \gamma 1_{v \in B(v_*, r_0)}. \]
\item[(R2)] $W \in C^2(\T^d)$ with $\nabla W \neq 0, V = \R^d$, the scattering function is bounded below by a decreasing radial function which is always strictly positive 
\begin{equation}
p(v,v') \geq M(|v|), \quad \forall v,v' \in V,  
\label{eq:anisotropic scattering}
\end{equation} where $M(v) = e^{m(|v|^2)}$ for some decreasing and Lipschitz continuous function $m: \mathbb{R} \rightarrow \mathbb{R}$.
\end{description} 

\begin{remark}
We briefly mention that the case $(R1)$ above does not allow us to consider velocities on the unit sphere and this situation is well treated by related works. This is because of a technical barrier relating to the way we estimate lower bounds on the solution $f$ in the proof of the main theorem. We will explain in more detail in another remark after Lemma 1.
\end{remark}

\par  \vspace{0.5em}

In what follows we consider measure-valued solutions to (\ref{eq:main}) and we refer to Definition \ref{defn:measure valued linear Boltzmann} for details. We denote by $\mathscr{M}(\T^d \times V)$ the space of measures on $\T^d \times V$, which is a Banach space endowed with the \emph{total variation} norm, denoted $\| . \|_{TV}$ (see (\ref{eq:total variation}) for details). We denote $\mathscr{P}(\T^d \times V)$ (respectively $\mathscr{P}(V)$) the space of probability measures on $\T^d \times V$ (respectively on $V$). Finally, for a given potential $W \in C^1(\T^d)$ we denote by $(T_t)_{t\geq 0}$ the transport semigroup generated by the corresponding characteristic flow and for the sake of notation, we drop the explicit dependence on $W$.  
\begin{defn} \label{def:free transport}
The transport semigroup on $\mathscr{P}(\Omega \times V)$, noted $(T_t)_{t \geq 0}$, is defined by
\begin{equation*}
(T_t \mu_0)(\phi) = \iint_{\Omega \times V} \phi(\Phi_{t}(x,v)) \ud \mu_0 ( \ud x, \ud v),  \quad \forall \phi \in C_b(\Omega \times V),
\end{equation*} for any $\mu_0 \in \mathscr{P}(\Omega \times V)$ and $t\geq 0$.
\end{defn}  

\par 

\vspace{0.5em}
Before we state our main theorem we mention that since we deal with measures many times we have chosen to use compact notation so $\mu \leq \nu$ means $\mu(A) \leq \nu(A)$ for every $A$ in the $\sigma$-algebra or equivalently $\int \phi \mathrm{d}\mu \leq \int \phi \mathrm{d}\nu$ for every positive continuous bounded function $\phi$.

Now, let $V \subseteq \R^d$ be an open set, let $p$ be a scattering function on $V \times V$, let $W\in C^1(\T^d)$ be a given potential function and let $\sigma \in C^0(\T^d)$. Later we will prove Proposition \ref{prop:duhamel} showing that the unique measure-valued solution $(\mu_t)_t$ to the linear Boltzmann equation (\ref{eq:main}) is global in time. Our main result is the following. 

\begin{thm}\label{thm:abstract}
Let $(\mu_t)_t$ be a measure-valued solution to the linear Boltzmann equation (\ref{eq:main}) and let $T_t$ the transport semigroup associated to the potential $W\in C^1(\T^d)$. Assume that 
\begin{itemize}
\item[(i)] $\sigma \in \C^0(\T^d)$ satisfies the GCC (Definition \ref{def:GCC2}) for some time $T$ with a constant $\kappa$,

\item[(ii)] there exists a $u_2 \in \mathscr{P}(V)$ such that $ p(\cdot, v')\geq \beta_2 u_2$, and for this $u_2$ there exist $T_*>0$, $\beta_1 >0$ such that  
\begin{equation}
 \inf_{x_0 \in \T^d } \int_V T_t \left(\delta_{x_0} \otimes u_2 \right) \ud v  \geq \beta_1 u_1,    \qquad  
 \label{hyp:t}
\end{equation} for every $t \in [T_*, T_* +T]$ and $v'\in V$,
\item[(iii)] for $u_2$ in the previous condition and recalling that $u_1$ is the uniform measure on the torus, there exists $\beta_3>0$ such that 
\begin{equation}
T_t (u_1 \otimes u_2) \geq \beta_3 u_1 \otimes u_2, \qquad \forall t \in \R.
\label{hyp:u}
\end{equation}
\end{itemize} Then there exists a unique equilibrium state $\nu \in \mathscr{P}(\T^d \times V)$ for (\ref{eq:main}) and
\begin{equation}
\| \mu_t - \nu \|_{TV} \leq e^{-\lambda(t-2T-T_*)}\| \mu_0 - \nu \|_{TV}, \qquad \forall t \geq 2 T + T_*,
\label{eq:exponential decay thm 1}
\end{equation} with the quantitative rate
\begin{equation}
\lambda = - \frac{1}{2T+T_*} \log\left(1-\beta_1 \beta_2^2 \beta_3  \kappa^2 e^{- (2T+T_*)\|\sigma\|_{\infty} }\right).
\label{eq: quantitative lambda 1}
\end{equation} 
\end{thm}
\begin{remark}
Condition \ref{hyp:t} is a condition linking the possible post collision velocities and the transport map, in the sense that it gives a quantitative estimate of the fact that for every point $y \in \mathbb{T}^d$ there is a velocity in the range of $p$ such that $x_0$ will be mapped to $y$ by following the transport map with that velocity. Condition \ref{hyp:u} is less restrictive. We expect it to hold very generally (with similar proofs to the specific situations detailed in this paper e.g. Lemma \ref{lem:T4}) but it is helpful to have the quantitative constant in the arguments which follow.
\end{remark}

The lower bound in (\ref{hyp:t}) is a crucial hypothesis intimately linked to Doeblin's theorem and is key to obtain the exponential rate (\ref{eq: quantitative lambda 1}), as can be seen in Section \ref{sec:Doeblin}. \par 

\vspace{0.5em}

As a consequence of Theorem \ref{thm:abstract} we obtain explicit decay rates for the linear Boltzmann equation in the particular regimes described by \textbf{(R1)} and \textbf{(R2)}. In the first instance we have the following result. 

\begin{cor}\label{thm:main}
Let $\sigma \in C^0(\T^d)$ satisfy definition (\ref{def:GCC1}) and assume that \textbf{(R1)} holds. Then, if $(\mu_t)_{t \geq 0}$ is a measure solution to \eqref{eq:main} with initial datum $\mu_0 \in \mathscr{P}(\T^d \times V)$, if we write $T_* =  \frac{2\sqrt{d}}{r_0}$ we will have
\begin{equation}
\| \mu_t - \nu \|_{TV} \leq e^{-\lambda(t-2T-T_*)}\| \mu_0 - \nu \|_{TV}, \qquad \forall t \geq T_*,
\label{eq:exponential decay thm R1}
\end{equation} with the quantitative rate
\begin{equation}
\lambda = - \frac{1}{2T  + \frac{2\sqrt{d}}{r_0} } \log\left(1- \frac{\gamma^2 |B(v_*,r_0)|^2}{2^d}  \kappa^2 e^{- (2T+ \frac{2\sqrt{d}}{r_0})\|\sigma\|_{\infty} }\right).  
\label{eq: quantitative lambda R1}
\end{equation} 
\end{cor}

In order to refine the quantitative bound in (\ref{eq: quantitative lambda 1}), we give in Lemma \ref{lem:T1} some sufficient conditions on $V$ so that (\ref{hyp:t}) holds with concrete choices of $\beta_1, \beta_2, \beta_3$ and $T_*$. \par 

\vspace{0.5em}
%
%
%
%
%
%
%
%
%
%

\vspace{0.5em}

Our second result concerns the regime \textbf{(R2)}, with non-zero potentials.

\begin{cor} \label{thm:BGK2}
Let $\sigma \in C^0(\T^d)$ satisfying definition (\ref{def:GCC2}) and assume that \textbf{(R2)} holds. Then, if $(\mu_t)_{t \geq 0}$ is a measure-valued solution to \eqref{eq:main} with initial datum $\mu_0 \in \mathscr{P}(\T^d \times V)$, then there exists a $T_{**}>0$ and $\beta_{**}>0$ that we can make quantitative so that
\begin{equation}
\|\mu_t - \nu \|_{TV} \leq e^{-\lambda(t-2T-T_{**} )}   \|\mu_0 - \nu\|_{TV}, \quad \forall \, t \geq 0, 
\label{eq:exponential decay thm 3}
\end{equation} with the quantitative rate
\begin{equation}
\lambda = -\frac{1}{  2T+T_{**}  } \log \left( 1- \beta_{**} \kappa^2 e^{- (2T+T_{**}) \|\sigma\|_\infty } \right). 
\label{eq: quantitative lambda 3}
\end{equation}
\end{cor}

\begin{remark}
Observe that Corollary \ref{thm:BGK2} contains quantitative rates in terms of $\beta _{**}$ and $T, T_{**}$. We will give in Section \ref{sec:geometric assumptions} precise results with explicit rates and assumptions. It is difficult to make them as compact as in \ref{thm:main}
\end{remark}

\begin{remark}
Observe that we are assuming that $\sigma \in \C^0(\T^d)$ instead of just bounded and measurable. This is a technical assumption due to the fact that we are working with measured-valued solutions. See Section \ref{sec:Measured} for details. 

\end{remark}

\subsection{Previous works: Hypocoercivity, Doeblin's theorem and the geometric control condition}
\label{sec:main results}

\subsubsection{Hypocoercivity results when $\sigma$ is strictly positive}

Finding quantitative rates of convergence to equilibrium is a long-standing problem in kinetic theory. In the context of spatially inhomogeneous kinetic equations this is usually done using the tools of hypocoercivity, a name given by Villani in \cite{V09} to equations exhibiting convergence like $Ce^{-\lambda t}$ where $C \geq 1$. In the context of kinetic equations, hypocoercive behaviour  is typically found when considering spatially inhomogeneous equations where the dissipation of natural entropies vanishes on a large class of functions, the local equilibria, making it impossible to prove entropy-entropy production inequalities. Techniques to prove convergence for such equations based on hypoellipticity methods were developed in \cite{HN04, MN06, V09} as well as in many other works. \par 

When $\sigma$ is constant, equation (\ref{eq:main}) is a key example of a hypocoercive equation, shown to converge faster than any power of $t$ in $H^1$ norm in \cite{CCG03} using the framework of \cite{DV05}. It was then shown to converge exponentially fast to equilibrium in $H^1$ weighted against the equilibrium in \cite{MN06} and in $L^2$ weighted against the equilibrium in \cite{H07}. The convergence in weighted $L^2$ can also be seen as a result of the general theorem in \cite{DMS15}. There are several other works showing exponential convergence in various norms or for various more complex versions of this equation we mention in particular \cite{CCEY19} since this work uses Doeblin/Harris's theorem, which is also the tool we will apply to the spatially degenerate case.

\subsubsection{Hypocoercivity results when $\sigma$ can vanish}

The case where $\sigma = \sigma(x)$ is non constant and can vanish on areas of the spatial domain was first studied in \cite{BL55} although it is mentioned somewhat indirectly. This paper deals with non-equilibrium steady states for scattering operators and is a pioneering example of the use of probabilistic tools in statistical physics, but without quantitative rates. \par 

The more recent works on these spatially degenerate models was begun in \cite{DS09} where the authors study a model where $\sigma$ vanishes at a discrete set of points. In \cite{BS13b} Bernard and Salvarani showed that there are situations where the velocity space and form of $\sigma$ together mean that there is no exponential convergence towards equilibrium. On the other hand, Bernard and Salvarani proved in \cite{BS13b} that the solutions to (\ref{eq:main}) with $\Omega \times V = \T^d \times \mathbb{S}^{d-1}$ and $W = 0$ convergence to equilibrium exponentially in $L^1$ if and only if the support of $\sigma$ satisfies the geometric control condition of Definition \ref{def:GCC1}. This work is then extended in \cite{MK14} to give a more delicate sense of when exponential convergence to equilibrium will occur. The results in \cite{BS13b,MK14}, based on semigroup theory and abstract functional analysis, do not give quantitative rate of the convergence.   \par 

An equation related to (\ref{eq:main}), the $1d$ Goldstein-Taylor type model, has been studied in \cite{BS13c} where the authors do get explicit rates via comparing this equation to a damped wave equation for which explicit rates were obtained by Lebeau in \cite{LebeauOndes}. \par

The case where $V$ is unbounded is treated in \cite{HL15} by Han-Kwan and L\'eautaud, where the authors study linear Boltzmann type equations for a general class of collision operators and external confining potential terms on a closed, smooth, connected and compact Riemannian manifold $M$ (and in particular the torus). In this context, the authors identify geometric control conditions in the natural phase space $T^*M$ (similar to Definition \ref{def:GCC2} in the case $M = \T^d$) allowing to completely characterise the convergence to equilibrium and exponentially fast convergence to equilibrium for the corresponding linear Boltzmann equation. On the other hand, the techniques developed in \cite{HL15}, using phase-space and micro-local tools inspired from \cite{BLR91,LebeauOndes} do not give explicit rates of convergence. \par

In \cite{Helge} the kinetic Fokker-Planck case is studied and here it is shown that the GCC is not equivalent to exponential convergence to equilibrium.

\subsubsection{Doeblin's theorem}

	We use techniques which are inspired from Doeblin's theorem from Markov process theory (see \cite{H10} for a detailed exposition of this theorem). This theorem was used to show convergence to equilibrium for scattering equations in \cite{BL55}. It has been used several times to study convergence to equilibrium for kinetic equations in the context of Non-Equilibrium Steady States \cite{CELMM16} and is currently being used for studying the convergence to equilibrium for solutions of PDEs from mathematical biology. We mention in particular the works on the renewal equation \cite{G17}, and the neuron population model \cite{CY18}. This last paper contains a similar type of degeneracy to that studied in this work. In this context Doeblin's theorem and Harris's theorem have been extended to PDEs which do not conserve mass and/or have time-periodic limiting solutions rather than steady states, as in \cite{BCG17, BCG19}.

\subsubsection{The geometric control condition in control theory}. The geometric control condition mentioned in the previous section plays a fundamental role in the study of controllability and stabilisation properties of some linear PDEs, typically of hyperbolic type. The GCC condition was introduced in the seminal works \cite{RauchTaylor,BLR91,LebeauSchrodinger} in order to prove that the linear wave equation and the Schr\"odinger equation in a domain $\Omega \subset \R^d$, possibly with boundary, are exactly controllable from an open subset $\omega$ (or a subset of the boundary) as long as $\omega$ satisfies the geometric control condition. In \cite{BurqGerard1996} the GCC condition is proved to be necessary for the exact controllability of the wave equation. As for the stabilisation properties, the works \cite{BLR91, LebeauOndes, BurqGerard2017} prove that under the GCC condition one can expect an exponential trend to equilibrium for the wave equation with a localised damping, which is a crucial inspiration for the works \cite{BS13b,HL15} on the linear Boltzmann equation.

\subsection{Strategy and Outline }
 We prove Corollaries \ref{thm:main} and \ref{thm:BGK2}. As stated above the proof is based around Doeblin's theorem for Markov processes. The key element to executing a Doeblin argument is to find a time $t_*$ such that we can prove a lower bound on the solution of the equation at time $t_*$ which is independent of the initial condition. We give a detailed proof of this fact based on using Duhamel's formula. We then explain how this implies exponential convergence to equilibrium via Doeblin's theorem.

\subsubsection*{Acknowledgements} We would like to thank many people for some useful discussion. In particularly Jos\'{e} Ca\~nizo. We had useful discussions with Francesco Salvarani, Havva Yolda\c{s}, Chuqi Cao, Helge Dietert and Cl\'{e}ment Mouhot. The first author was supported by FSPM postdoctoral fellowship (between October 2018-July 2020)
and the grant ANR-17-CE40-0030 and then by a Leverhulme trust grant ECF-2021-134. Much of this was written while the first author was visiting the Hausdorff Research Institute for Mathematics on a Junior Trimester fellowship. We would like to thank them for their hospitality. The second author was supported by the ERC grant MAFRAN.

\section{Measured-valued solutions to the linear Boltzmann equation}
\label{sec:Measured}

Let us first define some notation in order to state our results. Given $(\mathcal{X},\Sigma)$ a measurable space, we denote by $\mathscr{M}(\mathcal{X})$ the set of Radon measures on $\mathcal{X}$. We denote by $\mathscr{P}(\mathcal{X})$ the set of probability measures on $\mathcal{X}$, i.e., all measures $\mu \in \mathscr{M}(\mathcal{X})$ satisfying $\mu (\mathcal{X}) = 1$ and $\mu(A) \geq 0$ for every measurable $A$. As usual the space $\mathcal{P}(\mathcal{X})$ is endowed with the weak topology, denoted $w-\mathscr{P}(\mathcal{X})$, induced by the family of semi-norms 
\begin{equation*}
\phi \mapsto \int_{\mathcal{X}} \phi(z)  \mu (\ud z), \qquad \forall \phi \in C_b(\mathcal{X}), 
\end{equation*} i.e., we are using test functions which are continuous and bounded on $\mathcal{X}$.  Recall that $\mu \in \mathscr{M}(\mathcal{X})$ is said to be non-negative whenever
\begin{equation}
\int_{\mathcal{X}} \phi(x) \mu(\ud z) \geq 0, \qquad \forall \phi \in C_b(\mathcal{X}; \R_+). 
\label{eq:positive measure}
\end{equation} The \textit{total variation} distance in $\mathscr{M}(\mathcal{X})$ is defined as usual as 
\begin{equation}
\| \mu \|_{TV} := \sup \left\{ \int_{\mathcal{X}} \phi(z) \mu( \ud z); \, \phi  \in C_b(\mathcal{X}) , \, \|\phi\|_\infty \leq 1   \right\}.
\label{eq:total variation}
\end{equation}

Consider next a phase space of the form $\mathcal{X} = \Omega \times V$, where $\Omega = \T^d$. If $\Sigma_{\Omega \times V}$ is the Borel $\sigma$-algebra on $\Omega \times V$, we denote by $\mathscr{L}_{\Omega \times V}$ the Lebesgue measure on $\Omega \times V$. If $A \in \Sigma_{\Omega \times V}$, we simply denote by $|A|$ the Lebesgue measure of $A$ if no confusion arises.

\subsection{Measure-valued solutions}

With the notation of the previous section, given $T>0$ and $\mu_0 \in \mathscr{P}(X \times V)$, we consider the transport equation 
\begin{equation} 
\left\{
\begin{array}{ll}
\partial_t \mu + v \cdot \nabla_x \mu -\nabla_x W \cdot \nabla_v \mu = 0, & \textrm{ in } (0,T) \times \Omega \times V, \\
\mu|_{t=0} = \mu_0, &  \textrm{ in } \Omega \times V.
\end{array} \right.
\label{eq:free transport} 
\end{equation}

\begin{defn}
A measure solution to (\ref{eq:free transport}) is an element of $C^0([0,T]; w - \mathscr{P}(\Omega \times V))$ (continuous funtions from $[0,T]$ to the space of probability measures endowed with the topology of weak convergence). We denote the solution $\mu_t = \mu_t( \ud x, \ud v)$, and it satisfies that for every $\phi \in C^1_c([0,T) \times \Omega \times V)$,
\begin{align*}
\int_0^T \iint_{\Omega \times V} \left(\partial_t \phi - v \cdot \nabla_x \phi + \nabla_x W \cdot \nabla_v \phi \right) \mu_t (\ud x \ud v) \ud t = - \iint_{\Omega \times V} \phi(0,x,v) \mu_0 (\ud x \ud v).
\end{align*}
\label{def:measure solutions} 
\end{defn} 

We can write any weak solution to (\ref{eq:free transport}) using the transport semigroup. In particular, $\mu_t = T_t \mu_0 (\ud x, \ud v)$ is a measure solution to (\ref{eq:free transport}).

In this article we work with the linear Boltzmann equation (\ref{eq:main}) in the sense of measures. Given $\mu \in \mathscr{P}(\Omega \times V)$ we set
\begin{align} 
m_{\sigma} \mu (\ud x, \ud v):= \sigma(x) \mu (\ud x, \ud v), \qquad L^+ \mu(v, \ud x) & := \int_V p(v, v') \mu (\ud x, \ud v'), 
\end{align} which are respectively the multiplication by $\sigma$ and the average in the variable $v \in V$. Given $\mu_0 \in \mathscr{P}(\Omega \times V)$ we set
\begin{equation} 
\left\{ \begin{array}{ll}
\partial_t \mu + v \cdot \nabla_x \mu - \nabla_x W(x) \cdot \nabla_v \mu = m_{\sigma} \left( L^+ \mu - \mu \right),  & \textrm{in } (0,T) \times \Omega \times V, \\
\mu|_{t=0} = \mu_0, & \textrm{in } \Omega \times V,
\end{array} \right.
\label{eq:main(notation section)}
\end{equation} which is a version of (\ref{eq:main}) for measured-valued solutions. 

\begin{defn}
A measure solution to (\ref{eq:main(notation section)}) is an element of $C^0([0,T]; w - \mathscr{P}(\Omega \times V))$, denoted $\mu_t = \mu_t(\ud x, \ud v)$, satisfying that for every $\phi \in C^1_c([0,T) \times \Omega \times V)$,
\begin{align*}
\int_0^T & \iint_{\Omega \times V} \left(\partial_t \phi - v \cdot \nabla_x \phi+ \nabla_x W \cdot \nabla_v \phi  + m_{\sigma}(\phi -\int_{V} p(v', v) \phi(x,v') \mathrm{d}v') \right) \mu_t(\ud x \ud v) \ud t \\ &=- \iint_{\Omega \times V} \phi(0,x,v) \mu_0 (\ud x \ud v).
\end{align*} 
\label{defn:measure valued linear Boltzmann}
\end{defn}

The following existence and uniqueness result is well-known and can be proved by a standard contraction mapping argument in the space  $C^0([0,T], \mathscr{P}-TV)$ as one would for Picard iteration for ODEs. The semingroup property, conservation of mass, continuous dependence on initial conditions etc. can then be proved directly from the solution and its representation. This procedure is standard (see for example chapter 21 of \cite{DautrayLions}).

\begin{prop} \label{prop:duhamel}
Given $T > 0$ and given $\mu_0 \in \mathscr{P}(\Omega \times V)$, there exists a unique measure-valued solution to (\ref{eq:main(notation section)}), namely $\mu_t = \mu_t(\ud x, \ud v)$. Moreover, this solution admits the representation 
\begin{equation}
\mu_t( \ud x, \ud v) = \exp \left( - \int_0^t \sigma(\Phi^X_{-t+s}(x,v))\ud s \right)(T_t \mu_0)(\ud x, \ud v) + S_t[ \mu](\ud x, \ud v) \label{eq:Duhamel } 
\end{equation} where $(T_t)_{t\geq 0}$ is given by Definition \ref{def:free transport} and
\begin{equation}
S_t [ \mu_t] (\ud x, \ud v) = \int_0^t \exp \left( -\int_s^t \sigma(\Phi^X_{-t+r}(x,v)) \ud r \right) (T_{t-s} m_{\sigma} L^+ \mu_s)(\ud x, \ud v) \ud s.
\label{eq:Duhamel source}
\end{equation} Denoting 
\begin{equation}
\mu_t (\ud x, \ud v) = \mathcal{P}_t \mu_0, \qquad t \geq 0, 
\label{eq:semigroup Boltzmann}
\end{equation} the family $(\mathcal{P}_t)_{t\geq 0}$ is a semigroup on $\mathscr{M}(\Omega \times V)$ enjoying the following properties
\begin{align}
& \| \mathcal{P}_t \mu_0 \|_{TV} = 1, & \forall \mu_0 \in \mathscr{P}(\Omega \times V), \label{eq:conservation of mass}\\
& \| \mathcal{P}_t \mu_0 - \mathcal{P}_t \ovl{\mu_0} \|_{TV} \leq  \| \mu_0 -  \ovl{\mu_0} \|_{TV}, & \forall \mu_0, \ovl{\mu_0} \in \mathscr{P}(\Omega \times V) \label{eq:contraction Boltzmann semigroup}.
\end{align}
\label{prop:well-posedness Boltzmann}
\end{prop}

\section{Proof of Theorem \ref{thm:main}}

The goal of this section is to prove Theorem \ref{thm:abstract} based on Doeblin's theorem.

\subsection{Propagation of lower bounds}

\begin{lemma}\label{lemmaimp(delta)} 
Suppose that the hypothesis of Theorem \ref{thm:abstract} are satisfied. Let $T_*,\beta_1,\beta_2$ be as in (\ref{hyp:t}) and let $\beta_3$ as in (\ref{hyp:u}). Let $\mu_t = \mu_t( \ud x, \ud v)$ be the solution to \eqref{eq:main} with initial datum 
\begin{equation}
\mu_0 = \delta_{x_0} \otimes \delta_{v_0}, 
\label{eq:initial Dirac}
\end{equation} for $(x_0,v_0) \in \T^d\times V$ given. Then, for $t = 2T + T_*$ we have
\begin{equation}
\mu_t(\ud x,\ud v) \geq  \beta_1 \beta_2^2 \beta_3 \kappa^2 e^{-(2T+T_*)\| \sigma \|_{\infty}}   u_1 \otimes u_2, \qquad \textrm{ in }\mathscr{M}(\T^d \times V).
\label{eq:minoration lemmaimp(delta)}. 
\end{equation}
\end{lemma}

\begin{proof} 

Using Duhamel's formula (\ref{eq:Duhamel }) we have that, for every $t \geq 0$,
\begin{align}
\mu_t ( \ud x, \ud v) &  = \exp \left( - \int_0^t \sigma(\Phi^X_{-t+s}(x,v))\ud s \right)(T_t \mu_0)( \ud x, \ud v) + S_t[\mu_t](\ud x, \ud v)  \label{eq:first lower bound} \\
		 &  \geq \exp \left( - \int_0^t \sigma(\Phi^X_{-t+s}(x,v))\mathrm{d}s \right)(T_t \mu_0)(\ud x,\ud v)  \nonumber \\
		 & \geq e^{- t \| \sigma \|_{\infty}} (T_t \mu_0)(\ud x,\ud v), \nonumber 
\end{align} as, according to (\ref{eq:Duhamel source}), 
\begin{equation*}
S_t[\mu_t](\ud x, \ud v)  \geq 0 \qquad \textrm{ in }\mathscr{M}(\Omega \times V).
\end{equation*} Substituting (\ref{eq:first lower bound}) into the second term in (\ref{eq:Duhamel }) we get 
\begin{align*}
\mu_t(\ud x, \ud v) &  \geq  \int_0^t \exp \left( - \int_s^t \sigma(\Phi^X_{-t+\tau}(x,v)) \ud \tau \right)(T_{t-s} m_{\sigma} L^+ \mu_s)(\ud x, \ud v) \ud s \\
& \geq   \int_0^te^{- (t-s) \| \sigma \|_{\infty}  } (T_{t-s} m_{\sigma} L^+ \mu_s)( \ud x, \ud v) \ud s \\
& \geq  e^{- t\| \sigma \|_{\infty}  } \int_0^t (T_{t-s} m_{\sigma} L^+ T_s \mu_0)(\ud x, \ud v) \ud s.
\end{align*} Now we can substitute this in a second time to get
\begin{equation}
\mu_t( \ud x, \ud v) \geq e^{- t \| \sigma\|_{\infty}  } \int_0^t \int_0^s (T_{t-s} m_{\sigma} L^+ T_{s-\tau} m_{\sigma} L^+ T_{\tau} \mu_0)(\ud x, \ud v) \ud \tau \ud s.
\label{eq:lower bound 2}
\end{equation} We notice that $T_t (\delta_{x_0} \times \delta_{v_0}) = \delta_{\Phi_t(x_0, v_0)} = \delta_{\Phi^X_t(x_0,v_0)} \otimes \delta_{\Phi^V_t(x_0,v_0)}$. Now using (\ref{eq:initial Dirac}) we may write 
\begin{align*}
T_{s- \tau} m_{\sigma} L^+ T_{\tau} \mu_0 & = T_{s- \tau} m_{\sigma} L^+ \left(  \delta_{\Phi^X_\tau(x_0,v_0)} \otimes \delta_{\Phi^V_\tau(x_0,v_0)} \right) \\
& = T_{s- \tau} m_{\sigma} \left(  p(\ud v, \Phi^V_\tau(x_0,v_0)) \delta_{\Phi^X_\tau(x_0,v_0)}(\ud x)  \right) \\
& = T_{s- \tau} \left( \sigma(x) \delta_{\Phi^X_\tau(x_0,v_0)}(\mathrm{d}x) p( \ud v, \Phi^V_\tau(x_0,v_0))  \right) \\
& = \sigma(\Phi^X_\tau (x_0,v_0))T_{s-\tau}\left(\delta_{\Phi^X_\tau(x_0,v_0)}(\mathrm{d}x) p( \ud v, \Phi^V_\tau(x_0,v_0))\right). 
\end{align*} Now assuming  that $ s- \tau \geq T_*$, the definition of $T_*$ in assumption (\ref{hyp:t}) gives
\begin{align*}
L^+T_{s-\tau} m_{\sigma} L^+ T_{\tau} \mu_0 & = L^+ \sigma(\Phi^X_\tau (x_0,v_0))T_{s-\tau}\left(\delta_{\Phi^X_\tau(x_0,v_0)}(\mathrm{d}x) p( \ud v, \Phi^V_\tau(x_0,v_0))\right) \\
& \geq  \beta_2\sigma(\Phi_{\tau}^X(x_0,v_0)) L^+ T_{s-\tau} (\delta_{\Phi_{\tau}^X(x_0,v_0)  } \otimes u_2  ) \\
& \geq \beta_2^2 \sigma(\Phi_{\tau}^X(x_0,v_0))  u_2 \int_V T_{s-\tau} (\delta_{\Phi_{\tau}^X(x_0,v_0)  } \otimes u_2  ) \ud v  \\
& \geq \beta_1 \beta_2^2 \sigma(\Phi_{\tau}^X(x_0,v_0)) u_1 \otimes u_2.
\end{align*} Hence,  
\begin{equation*}
m_{\sigma}L^+T_{s-\tau} m_{\sigma} L^+ T_{\tau} \mu_0 \geq \beta_1 \beta_2^2 \sigma(\Phi_{\tau}^X(x_0,v_0)) \sigma(x) u_1 \otimes u_2 .
\end{equation*} Now, using (\ref{hyp:u}) we have
\[ T_{t-s} m_{\sigma} L^+ T_{s-\tau} m_{\sigma} L^+ T_{\tau} \mu_0 = \beta_1 \beta_2^2 \beta_3 \sigma(\Phi^X_\tau(x_0,v_0))\sigma(\Phi^X_{-t+s}(x,v))u_1 \otimes u_2.\] 
Expanding this computation for any test function $\phi$ we have
\begin{align*}
\int\phi(x,v)  T_{t} \left( \sigma(x,v) u_1 \otimes u_2 \right)(\mathrm{d}x, \mathrm{d}v) & = \int \sigma(x,v) \phi(\Phi_t(x,v)) u_1 \otimes u_2 (\mathrm{d}x \mathrm{d}v) \\
& = \int \sigma( \Phi_t( \Phi_{-t}(x,v))) \phi(\Phi_t(x,v))(u_1 \otimes u_2)( \mathrm{d}x, \mathrm{d}v) \\
& = \int \sigma( \Phi_{-t}(x,v)) \phi(x,v)) T_t(u_1 \otimes u_2)( \mathrm{d}x, \mathrm{d}v)\\
& \geq \beta_3 \int \sigma( \Phi_{-t}(x,v)) \phi(x,v) (u_1 \otimes u_2)(\mathrm{d}x, \mathrm{d}v).
\end{align*} 
Now, taking $t = 2 T + T_*$ as in the statement and integrating (\ref{eq:lower bound 2}) with respect to $ \tau \in [0,T], s \in [T+T_*, 2T+T_*]$ we get
\begin{align*} 
\mu_t (\ud x, \ud v ) & \geq \beta_1 \beta_2^2 \beta_3 e^{-(2T+T_*)\| \sigma\|_{\infty} } \int_{T+T_*}^{2T+T_*} \int_0^T \sigma(\Phi^X_{-(t-s)}(x,v)) \sigma(\Phi^X_\tau(x_0,v_0))  u_1 \otimes u_2  \ud \tau \ud s  \\
& \geq  \beta_1 \beta_2^2 \beta_3 \kappa^2 e^{-(2T+T_*)\| \sigma \|_{\infty}}   u_1 \otimes u_2,
\end{align*}
whence (\ref{eq:minoration lemmaimp(delta)}) follows. Here we note that we used the GCC on both the forwards and backwards flow and this is possible since the map $\Phi_t$ is invertible and $\Phi_{t}^{-1} = \Phi_{-t}$.
\end{proof} 

\begin{remark}
We are able to see at this point why it is not possible to apply our techniques when $V$ is the unit sphere. This is because by dimensional concerns our estimate in this case would need to involve at least three iterations of Duhamel's formula. After doing this $\sigma$ and the transport map become entangled in a way which means we cannot use the GCC to get a lower bound.
\end{remark}

The next result is an extension of Lemma \ref{lemmaimp(delta)}, valid for Dirac masses, to any initial data that is a probability measure.

\begin{lemma}\label{lemmaimp}
Under the same hypothesis of Lemma \ref{lemmaimp(delta)}, let $\mu_0 \in \mathscr{P}(\T^d \times V)$ and let $\mu_t$ be the associated solution to \eqref{eq:main(notation section)}. Then, for $t = 2T + T_*$ we have
\begin{equation}
\mu_t(\ud x,\ud v) \geq \beta_1 \beta_2^2 \beta_3 \kappa^2 e^{- t\| \sigma \|_{\infty} }  u_1 \otimes u_2 \qquad \textrm{ in }\mathscr{M}(\T^d \times V).
\label{eq:minoration lemmaimp} 
\end{equation} 
\end{lemma}

\begin{proof}
Let $\mu_0$ and let $\mu_t$ as in the statement. 
According to (\ref{eq:semigroup Boltzmann}), we can write $ \mu_t = \mathcal{P}_t \mu_0 $. We claim that it suffices to prove that
\begin{equation}
\mu_t  = \iint_{\T^d \times V} \left(\mathcal{P}_t \delta_{x_0, v_0} \right) \mu_0(\udd x_0, \udd v_0).  
\label{eq:claim lemma imp}
\end{equation}
 We are writing this as we were trying to avoid introducing lots of notation from Markov process theory (in this case Markov transition kernels). We can define the integral above by duality.
\begin{align*}
\iint_{\T^d \times V} \phi(x, v ) \mu_t( \mathrm{d}x, \mathrm{d}v) = \iint_{\T^d \times V} \left( \iint_{\T^d \times V} \phi(x,v)  \left(\mathcal{P}_t \delta_{x_0, v_0} \right)( \mathrm{d}x , \mathrm{d}v )\right) \mu_0 ( \udd x_0, \udd v_0), \forall \phi \in C^0(\Omega \times V).
\end{align*}
We note here that $\iint_{\T^d \times V} \phi(x,v)  \left(\mathcal{P}_t \delta_{x_0, v_0} \right)( \mathrm{d}x , \mathrm{d}v )$ is a function of $(x_0, v_0)$ and it is continuous as $\mathcal{P}_t$ is a continuous map from the space of probability measures with the topology of weak convergence to itself. Therefore if $(x_n, y_n) \rightarrow (x_0, y_0)$ then we will have $\delta_{(x_n, y_n)} \rightarrow \delta_{(x_0, y_0)}$ weakly so $\mathcal{P}_t(\delta_{(x_n, y_n)}) \rightarrow \mathcal{P}_t(\delta_{(x_0, y_0)})$ weakly.

 If (\ref{eq:claim lemma imp}) holds, Lemma \ref{lemmaimp(delta)} implies
\begin{align*}
\mathcal{P}_t \mu & = \iint_{\T^d \times V} \left(\mathcal{P}_t \delta_{x_0,v_0}\right) \mu_0(\ud x_0, \ud v_0) \\
& \geq \beta_1 \beta_2^2 \beta_3 \kappa^2 e^{-t \|\sigma\|_{\infty}} \iint_{\T^d \times V}  u_1 \otimes u_2 \mu_0(\ud x_0, \ud v_0)  \\
& = \beta_1 \beta_2^2 \beta_3 \kappa^2 e^{-t \|\sigma\|_{\infty}} u_1 \otimes u_2. 
\end{align*} Next, in order to prove (\ref{eq:claim lemma imp}), we observe that it is sufficient to check that 
\begin{equation*}
\nu_t := \iint_{\T^d \times V} (\mathcal{P}_t \delta_{x_0,v_0})\mu_0(\mathrm{d}x_0, \mathrm{d}v_0)
\end{equation*} is indeed a measure-valued solution to (\ref{eq:main(notation section)}) with initial datum $\mu_0$, as uniqueness of solutions (Proposition \ref{prop:well-posedness Boltzmann}) would imply $\nu_t = \mu_t$ and a fortiori (\ref{eq:claim lemma imp}). \par

According to Definition \ref{defn:measure valued linear Boltzmann}, let $\phi \in C^1_c((0,T] \times \T^d \times V)$. As $\phi$ and $\nabla_{t,x} \phi$ are bounded and $p$ is $C^1$, then 
\begin{equation*}
P\phi =  \left(\partial_t \phi - v \cdot \nabla_x \phi+ \nabla_x W \cdot \nabla_v \phi  + m_{\sigma}(\phi -\int_{V} p(v', v) \phi(x,v') \mathrm{d}v') \right) \in C^1((0,T] \times \T^d \times V). 
\end{equation*} Then, using Fubini's theorem, 
\begin{align*}
& \int_0^T \iint_{\T^d\times V}  \left(\partial_t \phi - v \cdot \nabla_x \phi+ \nabla_x W \cdot \nabla_v \phi  + m_{\sigma}(\phi -\int_{V} p(v', v) \phi(x,v') \mathrm{d}v') \right) \nu_t(\udd x, \udd v)  \\
& \qquad = \int_0^T \iint_{\T^d \times V}  P \phi   \left( \iint_{\T^d \times V} \mathcal{P}_t \delta_{x_0,v_0} \mu_0 (\udd x_0, \udd v_0) \right)(\udd x, \udd v)  \\
& \qquad = \iint_{\T^d \times V} \left( \int_0^T   \iint_{\T^d \times V} P \phi  \left( \mathcal{P}_t \delta_{x_0,v_0} \right)(\udd x, \udd v) \right)  \mu_0 (\udd x_0, \udd v_0) \\
& \qquad =  - \iint_{\mathbb{T}^d \times V} \phi(0,x_0,v_0) \mu_0(\udd x_0, \udd v_0 ).
\end{align*} Note here that $ \iint_{\T^d \times V} P \phi  \left( \mathcal{P}_t \delta_{x_0,v_0} \right)(\udd x, \udd v) $ is a bounded measurable function of $(t,x_0,v_0)$ (as it is continuous) and we use Fubini's theorem to commute the order of integrals with this function as the integrand. We also note here that we are integrating against $\mu_0$ which has finite mass so the fact that $ \iint_{\T^d \times V} P \phi  \left( \mathcal{P}_t \delta_{x_0,v_0} \right)(\udd x, \udd v) $ is bounded implies it is in $L^1$.
\end{proof}

\subsection{Doeblin type argument and exponential decay}
\label{sec:Doeblin}

Now we want to conclude the proof of Theorem \ref{thm:main} using Doeblin's theorem, which states the following result, whose proof can be found for instance in \cite[Thm 2.1]{CanizoMischler}.
 
\begin{thm}[Doeblin]
Let $S : \mathscr{M}(\T^d \times V) \rightarrow \mathscr{M}(\T^d \times V)$ be a stochastic operator satisfying that there exist $0 < \alpha < 1$ and $ \eta  \in \mathscr{P}(\T^d \times V)$ such that
\begin{equation}
S\mu \geq \alpha \eta, \qquad \forall \mu \in \mathscr{P}(\T^d \times V). 
\label{eq:Doblin bound 1}
\end{equation} Then $S$ has a unique stationary state $\mu^* \in \mathscr{P}(\T^d \times V)$ which is exponentially stable, and more
generally
\begin{equation}
\forall k \in \N, \qquad \| S^k \mu_1 - S^k \mu_2 \|_{TV} \leq (1-\alpha)^k \| \mu_1 - \mu_2 \|_{TV},
\label{eq:Doblin bound 2}
\end{equation} for all $\mu_1, \mu_2 \in \mathscr{P}(\T^d \times V)$. 
\label{thm:Doblin}
\end{thm}

\begin{proof}[Proof of Theorem \ref{thm:main}]

Let $t_*=2T+T_*$ and set
\begin{equation*}
\alpha := \beta_1 \beta_2^2 \beta_3 \kappa^2 e^{- t_* \|\sigma\|_{\infty}  }.
\end{equation*} We note that $\alpha <1$ since conservation of mass implies the mass of a lower bound on $\mu_t$ cannot be larger than 1. Set 
\begin{equation*}
S = \mathcal{P}_{t^*}, \qquad \eta = u_1 \otimes u_2.
\end{equation*}  Now, thanks to (\ref{eq:minoration lemmaimp}), the lower bound (\ref{eq:Doblin bound 1}) holds and thanks to Doeblin's theorem (cf. Theorem \ref{thm:Doblin}) we know that a unique equilibrium $\nu$ exists and furthermore (\ref{eq:Doblin bound 2}) yields
\begin{equation}
\forall k \in \N, \qquad \| \mathcal{P}_{kt^*} \mu_0 - \nu \|_{TV} \leq (1-\alpha)^k \| \mu_0 - \nu \|_{TV},
\label{eq:Doblin bound 3}
\end{equation} for every $\mu_0 \in \mathscr{P}(\T^d \times V)$. Let $t > t_*$ and set $k \in \N$ be such that
\begin{equation*}
k < \frac{t}{t_*} \leq k + 1.
\end{equation*} Then, using (\ref{eq:Doblin bound 3}),
\begin{align*}
\| \mathcal{P}_{t} \mu_0 - \nu \|_{TV} & = \| \mathcal{P}_{t} \mu_0 - \mathcal{P}_{t} \nu \|_{TV} \\
& \leq  \| \mathcal{P}_{kt_*} \mu_0 - \mathcal{P}_{kt_*} \nu \|_{TV} \\
& \leq  (1-\alpha)^k \| \mu_0 - \nu \|_{TV} \\
& \leq \exp\left( \frac{t-t_*}{t_*} \log( 1 - \alpha) \right) \| \mu_0 - \nu \|_{TV},
\end{align*} where we have used that, thanks to the choice of $k$, 
\begin{equation*}
(k+1) \log(1 - \alpha) \leq \frac{t}{t_*}  \log(1 - \alpha).
\end{equation*} This gives (\ref{eq:exponential decay thm 1}) with the rate (\ref{eq: quantitative lambda 1}).

\end{proof}

\section{Quantitative decay estimates in the regimes \textbf{(R1)} and \textbf{(R2)}}
\label{sec:geometric assumptions}

In this section we explain how the situations described by \textbf{(R1)} and \textbf{(R2)} imply a quantitative lower bound of the form (\ref{hyp:t}). As a consequence, Theorem \ref{thm:abstract} imply Corollaries \ref{thm:main} and \ref{thm:BGK2}.

\subsection{Proof of Corollary \ref{thm:main}}

\begin{lemma}\label{lem:T1} Assume that assumption \textbf{(R1)} is satisfied for some $v_* \in \R^d$ and $\gamma, r_0>0$ given. Then, the lower bound (\ref{hyp:t}) holds with  
\begin{equation}
T_* = \frac{2\sqrt{d}}{r_0}, \qquad \beta_1 =  2^{-d}, \qquad \beta_2 = \gamma|B(v_*,r_0)|, \qquad u_2 = \frac{1}{|B(v_*,r_0)|}\mathds{1}_{B(v_*,r_0)} \ud v.
\end{equation} 
\end{lemma}

\begin{proof}[Proof of Lemma \ref{lem:T1}] 
 In order to verify the lower bound in (\ref{hyp:t}), let $x_0 \in \T^d$ and write
\begin{align*}
\int_{B(v_*,r_0)} T_t \left( \delta_{x_0} \otimes u_2 \right) \mathrm{d}v = & \frac{1}{|B(v_*,r_0)|} \int_{B(v_*,r_0)} T_t \left( \delta_{x_0} \otimes \mathds{1}_{v \in B(v_*, r_0)} \right) \mathrm{d}v \\
=& \frac{1}{|B(v_*,r_0)|} \int_{B(v_*,r_0)}  \delta_{x_0}(x-vt) \mathds{1}_{v \in B(v_*, r_0)} \mathrm{d}v \\
=& \frac{1}{t^d|B(v_*,r_0)|}  \int_{B(x-tv_*, tr_0)} \delta_{x_0}(y) \mathds{1}_{y \in B(x-tv_*, tr_0)} \mathrm{d}y \\
=& \frac{1}{t^d|B(v_*,r_0)|}  \mathds{1}_{ x_0 \in B(x -tv_*, tr_0)}.
\end{align*} Now we need to recall that $x \in \mathbb{T}^d$ so we want to understand this ball as $B(x-tv_*, tr_0) \subseteq \mathbb{T}^d$. Since we are interested in this as a distibution on $\mathbb{T}^d$ the easiest way is to look at it by integrating against an arbitrary smooth 1-periodic function on $\mathbb{R}^d$, namely $\phi$. In this next section let $Q(x, r)$ be the union of all the open hypercubes with integer vertices contained inside $\bar{B}(x,r)$ then
\begin{align*}
 \int_{\mathbb{R}^d} \phi(x)  \mathds{1}_{x \in B(x_0 +tv_*, tr_0)} \mathrm{d}x \geq &  \int_{\mathbb{R}^d} \phi(x) \mathds{1}_{x \in Q(x_0+tv_*, tr_0)} \mathrm{d}x \\
=& |Q(x_0+tv_*, tr_0)| \int_{\mathbb{T}^d}\phi(x)\mathrm{d}x.
\end{align*}
Now we can see that if $r>\sqrt{d}$ then $B(x,r) \setminus Q(x,r) \subset B(x,r) \setminus B(x,r-\sqrt{d})$. We have $|Q(x,r)| = |B(x,r)| - |B(x, r)\setminus Q(x,r)| \geq |B(x,r)| - |B(x,r) - B(x,r-\sqrt{d})| = |B(x,r-\sqrt{d})|$.  Consequently, if $tr_0 > \sqrt{d}$ we have
\[ |Q(x_0+tv_*, tr_0)| \geq |B(x_0+tv_*, tr_0-\sqrt{d})| = |B(0,1)| (tr_0- \sqrt{d})^d. \] This means that as a distribution on the torus
\[ \int_V T_t \left( \delta_{x_0} \otimes u_2 \right) \mathrm{d}v \geq \frac{|B(0,1)| (tr_0-\sqrt{d})^d}{t^d |B(v_*,r_0)|} = \left(1 - \frac{\sqrt{d}}{r_0 t}\right)^d . \] Therefore for 
\begin{equation*}
t \geq  \frac{2\sqrt{d}}{r_0} 
\end{equation*} we have that (recalling that $u_1$ is the uniform measure on the torus)
\[ \int_V T_t \left( \delta_{x_0} \otimes u_2 \right) \mathrm{d}v \geq 2^{-d} u_1. \]

\end{proof}

\begin{lemma}\label{lem:T5}
Assume that assumption \textbf{(R1)} is satisfied. Then, (\ref{hyp:u}) in the statement of Theorem \ref{thm:abstract} holds with $\beta_3=1$. 
\end{lemma}

\begin{proof}
According to \textbf{(R1)}, $W = 0$ in this case and therefore $(T_t)_{t\geq 0}$ reduces to the free transport semigroup, which preserves spatially homogeneous distributions. Hence, (\ref{hyp:u}) follows. 
\end{proof}

\begin{proof}[Proof of Corollary \ref{thm:main}]

As assumption \textbf{(R1)} is satisfied for some $v_* \in \R^d$ and $\gamma, r_0>0$ by hypothesis, thanks to Lemmas \ref{lem:T1} and \ref{lem:T5}, the assumptions (\ref{hyp:t}) and (\ref{hyp:u}) hold. Now, applying Theorem \ref{thm:abstract} we get exponential decay with the explicit rate (\ref{eq: quantitative lambda 1}). Using the values found in Lemmas \ref{lem:T1} and \ref{lem:T5}, we get

\begin{align*}
\lambda & = - \frac{1}{2T+T_*} \log\left(1-\beta_1 \beta_2^2 \beta_3  \kappa^2 e^{- (2T+T_*)\|\sigma\|_{\infty} }\right) \\
& = - \frac{1}{2T +  \frac{2\sqrt{d}}{r_0} } \log\left(1- \frac{\gamma^2 |B(v_*,r_0)|^2}{2^d}  \kappa^2 e^{- (2T+ \frac{2\sqrt{d}}{r_0})\|\sigma\|_{\infty} }\right),
\end{align*} which is (\ref{eq: quantitative lambda R1}).

\end{proof}

\subsection{Proof of Corollary \ref{thm:BGK2}}

\begin{lemma}\label{lem:T3}
For $W$ smooth, periodic and positive, and for $p(v',v) \geq M(|v|)$ for strictly positive, decreasing $M$, we can find  $T_{**}< \infty$ and $\beta_{**} \in (0,1)$ such that for all $t \in [T_{**}, T_{**}+T],$ where $T$ is the time from the GCC, we have
\begin{equation}
\inf_{x_0 \in \T^d } \int_{\mathbb{R}^d} T_t \left(\delta_{x_0} \otimes p(v', v) \right) \ud v  \geq \beta_{**} u_1.
\label{hyp:t2}
\end{equation} Here $T_{**}=1/2$ and 
\[ \beta_{**} =  \exp \left( -(T+1)\left(1+ \|\mbox{Hess}(W)\|_\infty\right)\right)M(4(1+\|\nabla W\|_\infty)+5\|\nabla W \|_\infty T). \]
\end{lemma}

\begin{proof}[Proof of Lemma \ref{lem:T3}]
The strategy of this lemma is to split a time $t \in [T, 1+T]$ into the form $s+r$ where $s \in [1/2,1]$ and $r \in [T-1/2, T]$. For the part of the transport semigroup corresponding to the time of length $r$ we show that if we start with sufficient mass in large velocities we will retain a large amount of mass in large velocities. More precisely mass cannot move to velocities more than $T \|\nabla_x W\|_\infty$ away from the starting velocity and since the torus is compact $\|\nabla_x W\|_\infty <\infty$. Then for the part of the transport semigroup corresponding to the time of length $s$ we approximate the Hamiltonian flow by free transport and use the fact that for sufficiently large velocities the free transport maps moves mass to all possible $x$.

In this proof we first write the short time estimate approximating by free transport, then the longer time estimate separately. We then put the two together at the end.

We begin by looking at short times. We can use a Taylor expansion to write
\begin{equation}\label{eq:approxT}
\Phi^X_{-t}(x,v) = x-vt +\frac{1}{2}t^2 R(x,v,t),
\end{equation}
Here $R(x,v,t)_i = \partial_{x_i} W(\phi^X_{-s}(x,v))$ for some $s \in [0,t)$. We want to consider this map as free transport plus a perturbation. If we start with sufficiently large velocities, and since $\nabla_xW$ is bounded the contribution from $vt$ will be much larger than the contribution from $\nabla_x W$.
We will first consider for some $0<R_1<R_2$ the marginal measure given by 
\[ \int_{\mathbb{R}^d} T_t \left( \delta_{x_0} \times 1_{R_1 \leq |v| \leq R_2} \right) \mathrm{d}v.\] We study this by integrating it against a test function. We choose a smooth, positive test function $\psi(x)$ which is a function on all of $\mathbb{R}^d$ which is $1$-periodic in every direction. The periodicity of $\psi$ allows us to capture the dynamics of $x$ and $v$ mixing with the $x$ variable on the torus. Therefore we have
\begin{align*}
\int_{\mathbb{T}^d} \int_{\mathbb{R}^d} \psi(x) T_t\left( \delta_{x_0} \times 1_{R_1 \leq |v| \leq R_2} \right) \mathrm{d}v \mathrm{d}x =& \int_{\mathbb{T}^d} \int_{\mathbb{R}^d} \psi(x) \delta_{x_0}(\Phi^X_{-t}(x,v)) 1_{R_1 \leq |\Phi^V_{-t}(x,v)| \leq R_2} \mathrm{d}v \mathrm{d} x \\
=& \int_{\mathbb{T}^d} \int_{\mathbb{R}^d} \psi(\Phi^X_t(y,u)) \delta_{x_0}(y) 1_{R_1 \leq |u| \leq R_2} \mathrm{d}u \mathrm{d}y \\
=& \int_{\mathbb{R}^d} \psi(\Phi^X_t(x_0,u))1_{R_1 \leq |u| \leq R_2} \mathrm{d}u.
\end{align*} We used here the change of variables $(y,u) = (\Phi^X_{-t}(x,v), \Phi^V_{-t}(x,v))$ which has Jacobian equal to 1. We now use equation \eqref{eq:approxT}, and the fact that $|\nabla_x W| \leq G$,  to see that for $t \in (1/2,1)$ we have
\begin{equation} \label{eq:rbound} 1_{R_1 \leq |u| \leq R_2} \geq 1_{R_1 +G \leq |\Phi^X_t(x_0,u)-x_0| \leq R_2/2 -G}. \end{equation} Expanding on this we have that if $|x-\Phi^X| \leq R_2/2 -G$ and $t^2|R| \leq t^2G$ then $t|u| = |x- \Phi^X - t^2 R/2| \leq R_2/2 \leq t R_2$, and similarly if $|x- \Phi^X | \geq R_1 +G$ then $t|u| \geq R_1 \geq t R_1$. We then substitute this in to get that
\begin{align*}
\int_{\mathbb{T}^d} \int_{\mathbb{R}^d} \psi(x) T_t\left( \delta_{x_0} \times 1_{R_1 \leq |v| \leq R_2} \right) \mathrm{d}v \mathrm{d}x \geq & \int_{\mathbb{R}^d} \psi(\Phi^X_t(x_0,u)) 1_{2R_1 +G \leq |\Phi^X_t(x_0,u)-x_0| \leq R_2/2 -G} \mathrm{d}u \\
=& \int_{\mathbb{R}^d} \psi(x) 1_{2R_1 +G \leq |x-x_0| \leq R_2/2 -G} \frac{1}{|\partial_u \Phi^X_t(x_0,u)|} \mathrm{d}x.
\end{align*} 
Now we need to bound the Jacobian appearing here, we recall that the system of equations definiting $\Phi^X, \Phi^V$ are
\[ \frac{\mathrm{d}}{\mathrm{d}t} \Phi^X_t = \Phi^V_t, \quad \frac{\mathrm{d}}{\mathrm{d}t} \Phi^V_t = - \nabla_x W(\Phi^X_t). \] We can differentiate with respect to $v$ to get,
\[ \frac{\mathrm{d}}{\mathrm{d}t} \partial_v \Phi^X_t = \partial_v \Phi^V_t, \quad \frac{\mathrm{d}}{\mathrm{d}t} \partial_v \Phi^V_t = - \mbox{Hess}(W)(\Phi^X_t)\partial_v \Phi^X_t. \] We can use this to get the differential inequality
\[ \frac{\mathrm{d}}{\mathrm{d}t} \left( |\partial_v \Phi^X_t|^2 + |\partial_v \Phi^V_t|^2\right) \leq \left(1+\|\mbox{Hess}(W)\|_\infty)\right) \left( |\partial_v \Phi^X_t|^2 + |\partial_v \Phi^V_t|^2\right). \] Therefore by Gr\"onwall's inequality we have
\[ \left( |\partial_v \Phi^X_t|^2 + |\partial_v \Phi^V_t|^2\right) \leq \exp \left( t (1+\|\mbox{Hess}(W)\|_\infty)\right) \left( |\partial_v \Phi^X_0|^2 + |\partial_v \Phi^V_0|^2\right). \] As $\partial_v \Phi^X_0 = 0$ and $\partial_v \Phi^V_0 =1$ therefore it follows that
\[ |\partial_v \Phi^X_{t}| \leq \exp \left( t(1+\|\mbox{Hess}(W)\|_\infty)\right).  \]
Now this gives the following lower bound
\[  \min_{x,v, t \in (0,1+T]} \frac{1}{|\partial_v \Phi^X_t(x,v)|} \geq \exp(-(T+1)(1+ \|\mbox{Hess}(W)\|_\infty)) =: \alpha,  \] and we choose $R_1, R_2$ so that $R_2/2-R_1-2G \geq 2$. This will mean that the anulus $\{x \,:\,R_1 +G \leq |x-x_0| \leq R_2/2 -G\}$ contains at least one unit square say with integer vertices $Q \subset \{x\,\;\,R_1 +G \leq |x-x_0| \leq R_2/2 -G\}$. Then we have
\begin{align*}
\int_{\mathbb{T}^d} \int_{\mathbb{R}^d} \psi(x) T_t\left( \delta_{x_0} \times 1_{R_1 \leq |v| \leq R_2} \right) \mathrm{d}v \mathrm{d}x \geq & \int_{1_{R_1 +G \leq |x-x_0| \leq R_2/2 -G}} \psi(x) \alpha \mathrm{d}x \\
\geq & \int_{Q} \psi(x) \alpha \mathrm{d}x \\
=& \int_{\mathbb{T}^d} \psi(x) \alpha \mathrm{d}x.
\end{align*} This means as measures on the torus, when $t \in (1/2,1)$ and $R_2/2-R_1-2G \geq 2$, we have that
\[ \int_{\mathbb{R}^d} T_t \left( \delta_{x_0} \times 1_{R_1 \leq |v| \leq R_2} \right) \mathrm{d}v \geq \alpha. \] 
Now we would like to get a similar result covering a much larger range of times. Before we do this we first show bounds on how the transport semigroup moves velocities, we show that if we start with large velocities after time $t$ we will still have mass in large velocities.  We can see that for any $x_0$, and $t \leq T$ then since $\Phi^V_t = v + t R(x,v,t)$, where $R(x,v,t)_i = \partial_{x_i}W(x,v,s)$ for some $s \in (0,t)$ we have
\[ 1_{R_3 \leq |v| \leq R_4} \geq 1_{R_3+GT \leq |\Phi^V_t(x_0,v)| \leq R_4-GT}. \] Therefore, taking another smooth, positive, bounded test function $\tilde{\psi}$ which is now a function of $x$ and $v$ and is still periodic in $x$ we have
\begin{align*} \int_{\mathbb{T}^d}&\int_{\mathbb{R}^d}\tilde{\psi}(x,v) T_t\left(\delta_{x_0} \times 1_{R_3 \leq |v| \leq R_4}\right) \mathrm{d}x \mathrm{d}v  = \int_{\mathbb{T}^d} \int_{\mathbb{R}^d} \tilde{\psi}(\Phi^X_t(x,v), \Phi^V_t(x,v)) \delta_{x_0}(x)1_{R_3 \leq |v| \leq R_4} \mathrm{d}x \mathrm{d}v \\
\geq & \int_{\mathrm{T}^d} \int_{\mathrm{R}^d} \tilde{\psi}(\Phi^X_t(x,v), \Phi^V_t(x,v)) \delta_{x_0}(x) 1_{R_3+GT \leq |\Phi^V_t(x,v)| \leq R_4-GT} \mathrm{d}x \mathrm{d}v \\
=& \int_{\mathbb{T}^d} \int_{\mathbb{R}^d} \tilde{\psi}(x,v) \delta_{x_0}(\Phi^X_{-t}(x,v))1_{R_3+GT \leq |v| \leq R_4-GT} \mathrm{d}x \mathrm{d}v.\end{align*} Here we used the transformation $(x,v) \rightarrow (\Phi^X_t(x,v), \Phi^V_t(x,v))$ first in one time direction and then backwards. Therefore we have for $t \leq T$ that as measures
\[ T_t \left( \delta_{x_0} \times 1_{R_3 \leq |v| \leq R_4 }\right) \geq \delta_{x_0}(\Phi^X_{-t}(x,v)) 1_{R_3+GT \leq |v| \leq R_4-GT}. \]
Now suppose we have $t \in [1/2, 1/2+T]$ we can write this as $t=s+r$ where $r \leq T$ and $s \in (1/2, 1)$ then we have
\begin{align*} T_t \left( \delta_{x_0} \times 1_{R_3 \leq |v| \leq R_4 }\right) =& T_s \left( T_r  \left( \delta_{x_0} \times 1_{R_3 \leq |v| \leq R_4 }\right) \right) \\
\geq & T_s \left( \delta_{x_0}(\Phi^X_{-r}(x,v)) 1_{R_3+GT \leq |v| \leq R_4 -GT} \right).\end{align*}
Now we want to do both these steps at the same time, first by using the semingroup property and rearranging
\begin{align*}
\int_{\mathbb{T}^d}& \int_{\mathbb{R}^d} \psi(x) T_t \left( \delta_{x_0} \times 1_{R_3 \leq |v| \leq R_4} \right) \mathrm{d}v \mathrm{d}x = \int_{\mathbb{T}^d} \int_{\mathbb{R}^d} \psi(\Phi^X_t(x,v)) \delta_{x_0}(x) 1_{R_3 \leq |v| \leq R_4} \mathrm{d}v \mathrm{d}x \\
=& \int_{\mathbb{R}^d} \psi(\Phi^X_t(x_0,v)) 1_{R_3 \leq |v| \leq R_4} \mathrm{d}v \\
=& \int_{\mathbb{R}^d} \psi(\Phi^X_s\left( \Phi^X_r(x_0,v), \Phi^V_r(x_0,v)\right) 1_{R_3 \leq |v| \leq R_4} \mathrm{d}v \\
=& \int_{\mathbb{T}^d} \int_{\mathbb{R}^d} ( \psi \circ \Phi^X_s )(x,v) T_r ( \delta_{x_0} \times 1_{R_3 \leq |v| \leq R_4})( \mathrm{d}x, \mathrm{d}v).
\end{align*}
This means we are now able to use the fact that as $r \leq T$
\[ T_r \left( \delta_{x_0} \times 1_{R_3 \leq |v| \leq R_4 }\right) \geq \delta_{x_0}(\Phi^X_{-r}(x,v)) 1_{R_3+GT \leq |v| \leq R_4-GT}. \]
We then substitute this in and continue to get
\begin{align*}
\int_{\mathbb{T}^d}& \int_{\mathbb{R}^d} \psi(x) T_t \left( \delta_{x_0} \times 1_{R_3 \leq |v| \leq R_4} \right) \mathrm{d}v \mathrm{d}x \geq \int_{\mathbb{T}^d} \int_{\mathbb{R}^d} ( \Phi^X_s \circ \psi)(x,v)\delta_{x_0}(\Phi^X_{-r}(x,v)) 1_{R_3+GT \leq |v| \leq R_4-GT}(\mathrm{d}x, \mathrm{d}v)\\
= & \int_{\mathbb{R}^d} \psi(\Phi^X_s \left( \Phi^X_r(x_0,v), \Phi^V_r(x_0,v) \right) 1_{R_3+GT \leq |\Phi^V_r(x_0,v)| \leq R_4-GT} \mathrm{d}v \\
\geq & \int_{\mathbb{R}^d}  \psi(\Phi^X_s \left( \Phi^X_r(x_0,v), \Phi^V_r(x_0,v) \right)1_{(R_3+GT)+G \leq |\Phi^X_t(x_0,v) -\Phi^X_r(x_0,v)| \leq (R_4-GT)/2 -G} \mathrm{d}v.
\end{align*} 
Here in the last line we used \eqref{eq:rbound}. Now let us write $F(v) = \Phi^X_t(x_0,v)$ and use the change of variables $x=F(v)$ then we have ($F^{-1}$ exists globally since $\partial_v F$ is bounded above and below by the argument given above),
\begin{align*}
\int_{\mathbb{T}^d} \int_{\mathbb{R}^d} \psi(x) T_t \left( \delta_{x_0} \times 1_{R_3 \leq |v| \leq R_4} \right) \mathrm{d}v \mathrm{d}x \geq & \int_{\mathbb{R}^d} \psi(x) \frac{1}{|\partial_u F(u)|} 1_{2(R_3+GT)+G \leq |x-\Phi^X_r(x_0, F^{-1}(x)| \leq (R_4-GT)/2 -G} \mathrm{d}x
\end{align*}
Now taking $\alpha$ from before and provided that $(R_4-GT)/2-2(R_3+GT)-2G \geq 2$ we will have as above that
\[ \int_{\mathbb{R}^d} T_t \left( \delta_{x_0} \times 1_{R_3 \leq |v| \leq R_4} \right) \mathrm{d}v \geq \alpha. \] We can choose specific values for $R_3,R_4$ we may as well choose $R_3=0$ and $R_4=4(1+G)+5GT$.

Lastly we want to extend from looking at anuluses to looking at $p(v',\cdot)$. We know that since $M$ is decreasing 
\[ p(v',v) \geq M(|v|) \geq M(R_4)1_{R_3 \leq |v| \leq R_4}.  \] Therefore,
\[ \int_{\mathbb{R}^d} T_t\left(\delta_{x_0} \times p(v',\cdot)\right) \geq M(4(1+G)+5GT) \int_{\mathbb{R}^d} T_t \left( \delta_{x_0} \times 1_{R_3 \leq |v| \leq R_4} \right) \mathrm{d}v \geq M(4(1+G)+5GT)\alpha. \] This concludes the proof.

\end{proof}

\begin{lemma}\label{lem:T4}
Assume that assumption \textbf{(R2)} is satisfied. Then, (\ref{hyp:u}) in the statement of Theorem \ref{thm:abstract} holds with 
\begin{equation*}
\beta_3 = \exp( - 4 \| m \|_{Lip} \| W \|_{\infty}  ), u_2 = M(|v|).
\end{equation*}
\end{lemma}

\begin{proof}
According to \textbf{(R2)}, $W \in C^2(\T^d)$ and $m \in Lip(\T^d)$. Therefore, the flow $(\Phi_t)_{t\geq 0}$ defined by (\ref{eq:characteristic equations}) satisfies the conservation law
\begin{equation*}
W(x) + \frac{1}{2}|v|^2 = W(\Phi_{-t}^X(x,v)) + \frac{1}{2}\left| \Phi_{-t}^V(x,v)  \right|^2, 
\end{equation*} for every $(x,v)\in \T^d \times \R^d$. Hence, 
\begin{equation*}
m\left( \left| \Phi_{-t}^V(x,v)  \right|^2 \right) \geq m(|v|^2) - 2\|m \|_{Lip} \| W - W(\Phi_{-t}^X(\cdot,v))  \|_{\infty} \geq m(|v|^2) - 4\|m \|_{Lip} \| W \|_{\infty}.
\end{equation*} Then we can take $u_2 = e^{-m(|v|)} \mathrm{d}v$. Finally, as  
\begin{equation*}
T_t(u_1 \otimes u_2) =  \exp\left(  m \left(   \left| \Phi_{-t}^V(x,v)  \right|^2   \right) \right),
\end{equation*} we deduce 
\begin{equation*}
T_t(u_1 \otimes u_2) \geq  \exp( - 4 \| m \|_{Lip} \| W \|_{\infty}  ) u_1 \otimes u_2.
\end{equation*}
We have now verified all the conditions of Theorem \ref{thm:abstract} to prove Corollary \ref{thm:BGK2}.
\end{proof}

\section{Comments on the rates}


Lastly we comment on the rates we get. For the main model our rate is 
\[ \lambda = - \frac{\log\left(1-\kappa^2 e^{-\|\sigma\|_\infty (2T+T_*)}/2\right)}{2T+T_*}.\] This is almost definitely not optimal. To the best of our knowledge the rate should vary quite strongly depending on the geometry. We can give a little bit of information about a bound on the spectral gap and examples of situations where the spectral gap is well below this bound. In \cite{HL14} the authors prove some results on the spectrum of this operator. Defining the constants 
\[ C^-_\infty = \sup_{T>0} \inf_{x,v} \frac{1}{T} \int_0^T \sigma(\Phi^X_t(x,v)) \mathrm{d}t , \quad C^+_\infty = \inf_{T>0} \sup_{x,v}  \frac{1}{T} \int_0^T \sigma(\Phi^X_t(x,v)) \mathrm{d}t,\] it is proven in \cite{HL14} that the essential spectrum of the linear Boltzmann operator lies in the strip $\{ z : C^-_\infty \leq \mbox{Re}(z) \leq C^+_\infty \}$. They also show that the spectrum is contained in a strip of the form $ \{ 0\leq \mbox{Re}(z) \leq L_\infty \}$,  where $L_\infty$ is related to the supremum of the collision kernel. We can give an upper bound on the spectral gap in total variation using a simple probabilistic argument. 
\begin{lemma}
Let $f_t$ be the unique measure valued solution to \eqref{eq:main} with initial data $f_0$. If there exists $\lambda>0, A>0$ such that for all $f_0$,
\[ \| f_t - \nu \|_{TV} \leq Ae^{-\lambda t}\| f_0 - \nu\|_{TV}, \] then $\lambda \leq C^+_\infty$ using the notation above. 
\end{lemma}
\begin{proof}
We use a stochastic process whose law follows the equation \eqref{eq:main}. Let us define a Poisson process with intensity $\|\sigma\|_\infty$ whose jump times are $I_1, I_2, \dots$ and generate a sequence of iid random variables $U_1, U_2, \dots$ who are uniform on $[0, \|\sigma\|_\infty]$. Then for $t \in [I_n, I_{n+1})$ we define $X_t = X_{I_n} +tV_{I_n}$ and $V_t = V_{I_n}$. Then when $t=I_{n+1}$ then if $U_n \leq \sigma(X_{I_{n+1}})$ we draw $V_{I_n}$ from the distribution $p(\cdot, V_{I_n})$ and if $U_n > \sigma(X_{I_n})$ then $V_{I_{n+1}} = V_{I_n}$. We can check that the law of $(X_t, V_t)$ will satisfy \eqref{eq:main} (which incidentally gives another proof of existence of solutions). The jumping process is a time inhomogeneous Poisson process with intensity $\sigma(X_t)$ and if we start at $(x_0, v_0)$ conditional on us not having jumped we have $\sigma(X_t) = \sigma( \Phi^X_t(x_0,v_0))$ therefore by standard result about inhomogeneous Poisson processes we have
\[ \| f_t - \nu \|_{TV} \geq \mathbb{P}( \mbox{jumped no times in time } t) = \exp\left( - \int_0^t \sigma(\Phi^X_s(x_0,v_0)) \mathrm{d}s \right). \] Fixing $\epsilon$ there exists $T(\epsilon)$ such that
\[ \sup_{x,v} \int_0^{T(\epsilon)}\sigma(\Phi^X_s(x,v))\mathrm{d}s \leq (C^+_\infty +\epsilon) T(\epsilon). \] Therefore, since we were taking the supremum over all $(x,v)$ we will have \[\int_{kT(\epsilon)}^{(k+1)T(\epsilon)} \sigma( \Phi_s (x,v)) \mathrm{d}s = \int_0^{T(\epsilon)} \sigma(\Phi_s( \Phi_{kT(\epsilon)}(x,v))) \mathrm{d}s \leq (C^+_\infty + \epsilon) T(\epsilon)\] and hence
\[ \| f(nT(\epsilon))- \nu \|_{TV} \geq \exp \left( - \int_0^{nT(\epsilon)} \sigma(\Phi^X_s(x,v)) \mathrm{d}s \right) \geq \exp\left(- nT(\epsilon)(C^+_\infty + \epsilon) \right), \] for every $n$. Therefore $\lambda \leq C^+_\infty + \epsilon$ and $\epsilon$ is arbitrary which gives the result.
\end{proof}

The consideration of optimal rates raises several natural further questions. The first is to investigate the optimal rates. Secondly it would be interesting to characterize which possible choices of $\sigma$ lead to the fastest and slowest rates. This is especially interesting since it is not obvious that having constant $\sigma$ gives the fasted rates, particularly in the presence of a confining potential. If it is possible to choose a degenerate $\sigma$ so that the convergence to equilibrium was much faster than the optimal choice of constant $\sigma$ then this could have implications for Hamiltonian Markov chain Monte-Carlo simulation. This is because discrete versions of this flow are used to sample from $e^{-W(x)}$ in HMCMC schemes and the intensity of the noise is generally chosen to be constant. If a spatially varying $\sigma$ could increase the convergence speed of this continuous flow the same might be true for the HMCMC schemes.

\bibliographystyle{plain}

 \end{document}